\documentclass[12pt,oneside,english,a4paper]{amsart}
\usepackage[T1]{fontenc}
\usepackage[latin9]{inputenc}
\usepackage{geometry}
\usepackage{refstyle}
\usepackage{hyperref}
\usepackage{amstext}
\usepackage{amsthm}
\usepackage{amssymb}
\usepackage{setspace}
\makeatletter

\AtBeginDocument{\providecommand\thmref[1]{\ref{thm:#1}}}
\AtBeginDocument{\providecommand\corref[1]{\ref{cor:#1}}}
\AtBeginDocument{\providecommand\lemref[1]{\ref{lem:#1}}}
\AtBeginDocument{\providecommand\claimref[1]{\ref{claim:#1}}}
\RS@ifundefined{subsecref}
  {\newref{subsec}{name = \RSsectxt}}
  {}
\RS@ifundefined{thmref}
  {\def\RSthmtxt{Theorem~}\newref{thm}{name = \RSthmtxt}}
  {}
\RS@ifundefined{lemref}
  {\def\RSlemtxt{Lemma~}\newref{lem}{name = \RSlemtxt}}
  {}
 \RS@ifundefined{corref}
  {\def\RScortxt{Corollary~}\newref{cor}{name = \RScortxt}}
  {}
  \RS@ifundefined{claimref}
  {\def\RSclaimtxt{Claim~}\newref{claim}{name = \RSclaimtxt}}
  {}
   \RS@ifundefined{propref}
  {\def\RSproptxt{Proposition~}\newref{prop}{name = \RSproptxt}}
  {}

\numberwithin{equation}{section}
\numberwithin{figure}{section}
\theoremstyle{plain}
\newtheorem{thm}{\protect\theoremname}
  \theoremstyle{remark}
  \newtheorem{rem}[thm]{\protect\remarkname}
  \theoremstyle{definition}
  \newtheorem{defn}[thm]{\protect\definitionname}
  \theoremstyle{plain}
  \newtheorem{cor}[thm]{\protect\corollaryname}
  \theoremstyle{plain}
  \newtheorem{lem}[thm]{\protect\lemmaname}
  \theoremstyle{plain}
  \newtheorem{prop}[thm]{\protect\propositionname}
  \theoremstyle{remark}
  \newtheorem{claim}[thm]{\protect\claimname}

\date{}
\usepackage{amsmath}
\usepackage{bbm}
\usepackage{dsfont}

\usepackage{tikz}

\usetikzlibrary{shapes.misc}
\usetikzlibrary{shapes.geometric}
\usetikzlibrary{arrows.meta}

\tikzset{cross/.style={cross out, draw=black, minimum size=2*(#1-\pgflinewidth), inner sep=0pt, outer sep=0pt}, cross/.default={1pt}}

\makeatother

\newcommand{\trans}{\mathcal{B}}
\newcommand{\clean}{\mathcal{C}}
\newcommand{\points}{\triangleright}
\newcommand{\susc}{\mathcal{S}}
\newcommand{\Z}{\mathbb{Z}}
\newcommand{\iid}{i.i.d.}
\newcommand{\rev}{rev}

\usepackage{babel}
  \providecommand{\claimname}{Claim}
  \providecommand{\corollaryname}{Corollary}
  \providecommand{\definitionname}{Definition}
  \providecommand{\lemmaname}{Lemma}
  \providecommand{\propositionname}{Proposition}
  \providecommand{\remarkname}{Remark}
\providecommand{\theoremname}{Theorem}

\begin{document}
\global\long\def\One{\mathds{1}}

\global\long\def\Laplacian{\Delta}

\global\long\def\grad{\nabla}

\global\long\def\norm#1{\left\Vert #1\right\Vert }

\global\long\def\zz{\mathbb{Z}}

\global\long\def\rr{\mathbb{R}}

\global\long\def\nn{\mathbb{N}}

\global\long\def\pp{\mathbb{P}}

\global\long\def\ee{\mathbb{E}}

\global\long\def\floor#1{\left\lfloor #1\right\rfloor }

\global\long\def\Intd#1#2#3#4{\int\limits _{#1}^{#2}#3\text{d}#4}

\global\long\def\intd#1#2#3#4{\int_{#1}^{#2}#3\text{d}#4}

\title{Time Scales of the Fredrickson-Andersen Model on Polluted $\zz^{2}$
and $\zz^{3}$}
\author{Assaf Shapira}
\email{assafshap@gmail.com}
\address{LPSM UMR 8001, Universit\'e Paris Diderot\protect\\CNRS, Sorbonne Paris Cit\'e}
\author{Erik Slivken}
\email{eslivken@gmail.com}
\address{CEREMADE UMR 7534, Universit\'e Paris-Dauphine\protect\\CNRS, PSL Research University}

\thanks{ This work has been supported by the ERC Starting Grant 680275 ``MALIG'', ANR-15-CE40-0020-01. 
 }

\maketitle

\section{Introduction}

The Fredrickson-Andersen $k$-spin facilitated model (FA$k$f) was
introduced in physics \cite{FA} in order to study the liquid-glass transition,
and it is a part of a family of interacting particle systems called
Kinetically Constrained Models (KCM). One may also see them as the stochastic
counterpart of a well known family of cellular automata called
Bootstrap Percolation (BP); and the Fredrickson-Andersen $k$-spin facilitated
model is the KCM corresponding to the $k$-neighbors bootstrap percolation. 

Both bootstrap percolation and the kinetically constrained models have been studied extensively in homogeneous environment, and in particular on the lattice $\zz^d$ (see, e.g., \cite{vanEnter87, AL, CerfManzo, Holroyd, BBDCM_SharpThresholdZd, BDCMS_universality2d, CMRT_kcmzoo, MT_towards_universality,MMT_universality2d}). Some results are known for the bootstrap percolation in random environments (e.g., \cite{JLTV_BPonER, BP_BPonRandomRegular, Janson_RandomRegular, BGCSP_BPonGW, Shapira_BPonGW, GravnerMcDonald_polluted, GravnerHolroyd_polluted}). However, kinetically constrained models in random environments have only been considered very recently in the mathematical literature \cite{Shapira_RandomConstraint}.

We will focus on the $2$-neighbors model on the polluted lattice, introduced
in \cite{GravnerMcDonald_polluted} for $\zz^{2}$ and recently analyzed on $\zz^{3}$ \cite{GravnerHolroyd_polluted,GravnerHolroydSivakoff_polluted3d}.
We study the divergence of time scales in this model for the stationary settings.

\section{Model, notation and main result}

We consider the random environment giving each $x\in\zz^{d}$ a quenched
variable 
\[\omega_{x}\in\left\{ \text{susceptible, immune}\right\}.
\]
These variables are
chosen in the beginning according to a measure $\nu$, which is the
product of $\iid$ Bernoulli random variables where $\nu(\omega_x = \text{immune}) = \pi$ and $\nu(\omega_x = \text{susceptible}) = 1-\pi$ for some fixed $\pi\in\left(0,1\right)$.  Once the environment is fixed, the stochastic dynamics will take place on the subset of susceptible sites, $\susc\subset \zz^d$.  

Susceptible sites with have one of two states: infected or healthy. The stochastic dynamics is defined over configurations $\Omega=\{i,h\}^\susc,$ where $i$ corresponds to an infected site, and $h$ corresponds to a healthy site.  We will denote such configurations by $\eta\in\Omega.$  We may wish to specify the state after changing the configuration on a set of sites $X$.  For $a \in \{i,h\}$, and $\eta\in \Omega$, let $\eta^{X\leftarrow a}$ denote the configuration which agrees with $\eta$ on all sites in $\susc \setminus X$ and equals $a$ on every site in $X$. For brevity, when $X$ is a single site $x$ we let $\eta^{x\leftarrow a} = \eta ^{X\leftarrow a}.$  We let $\eta^x$ denote the configuration that agrees with $\eta$ on $\susc\setminus \{x\}$ and differs from $\eta$ at $x$.  For a function $f:\Omega\rightarrow\rr$, $\grad_{x}f(\eta)=f\left(\eta^{x}\right)-f\left(\eta\right)$.

The FA dynamics will be chosen to be reversible with respect to a product measure
$\mu$, giving each site probability $q$ to be infected and $1-q$ to be
healthy for a small parameter $q\in (0,1)$. We will often take the expectation of function with respect to a single site $x$, which we dente by 
\[
\mu_{x}(f)=qf\left(\eta^{x\leftarrow{i}}\right)+\left(1-q\right)f\left(\eta^{x\leftarrow{h}}\right).
\]

In order to define the FA dynamics we need to define the constraints for $x\in \susc$:
\begin{equation}
c_{x}(\eta)=\begin{cases}
1 & x\text{ has at least }2\text{ susceptible infected neighbors}\\
0 & \text{otherwise}
\end{cases}.\label{eq:constraints}
\end{equation}
The dynamics will then follow the following rules \textendash{} each
site rings at rate $1$. If the constraint is satisfied (i.e. $c_{x}=1$)
we toss a coin (independently of everything) that gives $i$
with probability $q$ and $h$ with probability $1-q$.  Then set the state of $x$ to
the result of the coin toss. This could be equivalently described \cite{Liggett}
by the generator of the Markov semi-group defined by 
\[
\mathcal{L}f(\eta)=\sum_{x\in\susc}c_{x}(\eta)\mu_{x}f,
\]
where $f:\Omega\rightarrow\rr$ is a local function, i.e. depends
on the state of finitely many sites.  The Dirichlet form corresponding to $\mathcal{L}$ is 
\[
	\mathcal{D} f = -\mu(f \mathcal{L} f) = -\int f(\eta) \mathcal{L} f(\eta) d\mu(\eta) = q(1-q)\int \sum_{x\in \susc} c_x(\eta) (\grad_x f(\eta)) ^2 d\mu(\eta).
\]

Probabilities and expected values
with respect to this process starting at some initial state $\eta$
will be denoted by $\pp_{\eta}$ and $\ee_{\eta}$, and when starting
from equilibrium by $\pp_{\mu}$ and $\ee_{\mu}$. Though not mentioned
explicitly in the notation, these measures depend on the quenched variables, $\omega$, which describe the disorder.

Bootstrap percolation is deterministic in discrete time.  At each step $t$, sites that satisfy the constraint ($c_{x}(\eta_t)=1$)
get infected, and remain so forever.  Sites that would never be infected under the bootstrap percolation dynamics will never change their state under the corresponding KCM dynamics with the same starting configuration. 

\begin{rem}
The terminology and notation used by the KCM community is not the
same as that of the BP community, e.g., ``occupied'' and ``empty''
have an inverse meaning, as well as the labels 0 and 1. Here we chose
to use the more neutral terminology ``infected'' and ``healthy'',
hoping it will be equally confusing for readers of all backgrounds.
\end{rem}

As $q\rightarrow 0$, more and more sites are healthy, the constraint
is more difficult to satisfy, and the dynamics slows down. In order
to quantify this slowing down we should study typical time scales
of the system. One option is studying the spectral gap of the generator
(e.g. \cite{CMRT_kcmzoo}), which gives a lot of information on many time scales
of the system, and in particular the loss of correlation. In disordered
systems, however, the spectral gap tends to focus on ``bad'' parts
of the environment, giving an overly pessimistic estimation which
does not describe actual time scales of the system \cite{Shapira_RandomConstraint}.

One is then tempted to try to hide these bad regions, e.g., by removing
them from the graph and replace them with entirely healthy boundary
conditions. This choice should give, in a sense, a dynamics which
is the slowest possible (as more sites will not satisfy the constraints in this setting); and for the other bound we may take entirely
infected boundary conditions (possibly increasing the number of sites which satisfy the constraints). Unfortunately, KCMs are not attractive,
and a monotone coupling of the dynamics with hidden parts and the
original one is not possible.

In fact, the only information we can gain from such a coupling is
the finite speed of propagation of information \textendash{} if we
are interested in the process until time $t$, we may change the environment
at distances greater than $100t$ without effecting the dynamics near
the origin (see, e.g., \cite[section 3.3]{Martinelli_SaintFlourNotes}). In the analysis of the FA$1$f model on polluted $\zz^{2}$
this coupling indeed allows us to hide bad areas \cite[section 3.6]{Shapira_thesis}. However,
for the FA$2$f model this is impossible \textendash{} at distances
of the time scales we are considering, the system is not ergodic, even for entirely
infected boundary conditions.  For example, in the case of $\zz^{2}$, at distance of order $\pi^{-4}$ from the origin we will find four
corners of a rectangle that are all immune. If in the initial configuration
all sites in this rectangle are healthy, none could ever be infected,
and we cannot hope for correlations to be lost. We will see in \thmref{z2}
that the typical time scale for the evolution of the system is much
longer than $\pi^{-4}$, so we will not be able to use an argument based on the finite speed of propagation.

The way we approach this problem is by studying the infection time of
the origin. Unlike the spectral gap, this is a concrete observable,
so it will be affected by far away regions only to the extent that
the observed dynamics depends on them in practice. Moreover, the reversibility
of the process gives tools that allows us to study the Poisson problem
related to this time.

We therefore define the main quantity of this paper 
\[
\tau_{0}=\inf\left\{ t\,|\,\eta_{0}(t)=i\right\} .
\]

In the two dimensional case, \cite{GravnerMcDonald_polluted} show that when $q$ is small and
$\pi<cq^{2}$ the probability that the origin is eventually infected
is big, but when $\pi>Cq^{2}$ it is small. We will thus concentrate
on the case $\pi<cq^{2}$, taking some margins that will simplify
the analysis.
\begin{thm}
\label{thm:z2}Consider the FA2f model on $\zz^{2}$ with $\pi<q^{2+\varepsilon}$.
Then for all $\varepsilon>0$, with $\nu$-probability at least $1-5q^{\varepsilon/6}$,
$\omega$ is such that
\begin{equation*}
\pp_{\mu}\left(\tau_{0}>e^{-q^{-1-\varepsilon}}\right)  \le 3q^{\varepsilon/12}.
\end{equation*}
Moreover, in the other direction, uniformly in $\omega$, there exitst $C$ such that
\begin{equation*}
\pp_{\mu}\left(\tau_{0}<e^{-Cq^{-1}}\right) \xrightarrow{q\rightarrow0} 0	
\end{equation*}

\end{thm}

For polluted environments in $\zz^3$, it is shown in \cite{GravnerHolroyd_polluted} that for $\pi$
small enough (but not going to $0$ with $q$) the BP infects the origin with high probability even when $q$ tends to $0$.  
\begin{thm}\label{thm:z3}
Consider the FA2f model on polluted $\zz^{3}$.  For all $\varepsilon>0$,
with $\nu$-probability that tends to $1$ as $\pi\rightarrow 0$ uniformly
in $q$, $\omega$ is such that
\begin{align*}
\pp_{\mu}\left(\tau_{0}>e^{-q^{-1-\varepsilon}}\right) &= o(1)
\end{align*}
In the other direction, uniformly in $\omega$, there exitst $C$
such that
\[
\pp_{\mu}\left(\tau_{0}<e^{-Cq^{-1/2}}\right)\xrightarrow{q\rightarrow0}0
\]
\end{thm}

\begin{rem}
The exponents in the upper and lower bounds of the theorem above do not match. The reason is that the proof of upper bound uses an infection mechanism that takes place in a two dimensional surface, and indeed the power $q^{-1}$ fits the scaling of the two dimensional FA model. We conjecture, however, that in the true dynamics infection would be able to propagate in three directions, giving $q^{-1/2}$ as in the lower bound (perhaps up to logarithmic corrections).
\end{rem}

\begin{rem}
We mention for comparison the scaling of $\tau_{0}$ in the homogeneous
(non-polluted) model. In $\zz^{2}$ it scales (up to log corrections)
as $e^{1/q}$, and in $\zz^{3}$ as $e^{1/\sqrt{q}}$.
\end{rem}

\section{Preparation}

In order to prove the upper bounds we fix some high probability
event $E$, and show that the process cannot spend a lot of time in
$E$ before hitting $\left\{ \eta_{0}=0\right\} $. Since $E$ has
high probability, the process spends a lot of time in $E$ and therefore $\tau_0$ cannot be too big.  This entire section will assume the pollution, $\omega$, to be fixed.

Fix an event $E\subseteq\Omega$ and $t>0$. We will define the time spent
in $E$ by time $t$ as
\begin{equation}
T_{t}^{E}=\int_{0}^{t}\One_{E}\left(\eta\left(s\right)\right)\text{d}s.\label{eq:timespentinE}
\end{equation}
With some abuse of notation we will also considered its averaged version
\begin{equation}
T_{t}^{E}\left(\eta\right)=\ee_{\eta}\left[T_{t}^{E}\right]\label{eq:timespent_averaged}
\end{equation}
where we recall that $\ee_\eta[\cdot]$ is the expectation over the stochastic process starting from the configuration $\eta$.

For some event $A\subseteq \Omega$, $\tau_A$ denote the hitting time for this event.  

\begin{defn}
\label{def:timeinEbeforeA}Let $E,A\subseteq\Omega$ be two events.
The time spent in $E$ before hitting $A$ is
\[
T_{A}^{E}=T_{\tau_{A}}^{E}.
\]
\end{defn}

Also for $T_{A}^{E}$ we define
\begin{equation}
T_{A}^{E}\left(\eta\right)=\ee_{\eta}\left[T_{A}^{E}\right].\label{eq:timeinEbeforeA_averaged}
\end{equation}
Recall that $\mathcal{L}$ is the generator of the Markov process.  The function, $T_A^E$, solves the Poisson problem (see, e.g., \cite[equation (7.2.45)]{BovierDenHollander})
\begin{equation}
\begin{array}{cccc}
\mathcal{L}T_{A}^{E}\left(\eta\right) & = & -\One_{E}\left(\eta\right) & \quad\eta\notin A,\\
T_{A}^{E}\left(\eta\right) & = & 0 & \quad\eta\in A.
\end{array}\label{eq:poissonproblem_EbeforeA}
\end{equation}
Multiplying both sides by $T_{A}^{E}$ and integrating with respect
to $\mu$ gives
\begin{cor}
\label{cor:dirichletequalsexpectation_EbeforeA}$\mu\left(T_{A}^{E}\One_{E}\right)=\mathcal{D}\left(T_{A}^{E}\right).$
\end{cor}

We will use this formula in order to bound $\mu\left(T_{A}^{E}\One_{E}\right)$.
\begin{lem}
\label{lem:path_to_time}Fix $E,A\subseteq\Omega$, $N,D,V\in\nn$.
Assume that for every $\eta\in E$ there exists a sequence $\eta_{0},\dots,\eta_{N}$
of configurations and a sequence of sites $x_{0},\dots x_{N-1}$ such
that
\begin{enumerate}
\item $\eta_{0}=\eta$,
\item $\eta_{N}\in A$,
\item $\eta_{i+1}=\eta_{i}^{x_{i}}$ or $\eta_{i+1}=\eta_{i}$,
\item $c_{x_{i}}(\eta_{i})=1$,
\item for all $i\le N$, $\eta_{i}$ differs from $\eta$ on a set $X$ whose size is at most $D$, and $X$ is contained in a set $Y(x_i)$, depending only on $x_i$, whose size is at most $V$.
\end{enumerate}
Then $\mu\left(T_{A}^{E}\One_{E}\right)\le N^{2}\binom{V}{D}2^{D}q^{-D-1}$.
\end{lem}

\begin{proof}
Consider $\eta\in E$. By (2), $T_{A}^{E}\left(\eta_{N}\right)=0$, thus,
denoting $\grad_{i}T_{A}^{E}=T_{A}^{E}\left(\eta_{i+1}\right)-T_{A}^{E}\left(\eta_{i}\right)$
\[
T_{A}^{E}\left(\eta\right)=\sum_{i=0}^{N-1}c_{x_{i}}\left(\eta_{i}\right)\grad_{i}T_{A}^{E}.
\]
Note that $\grad_{i}T_{A}^{E}$ is either $\grad_{x_{i}}T_{A}^{E}\left(\eta_{i}\right)$
or $0$ (since we allow empty moves). We can then write
\begin{align*}
\mu\left(T_{A}^{E}\One_{E}\right)^{2} & \le\mu\left(\left(T_{A}^{E}\One_{E}\right)^{2}\right)=\sum_{\eta\in E}\mu\left(\eta\right)\left(\sum_{i=0}^{N-1}c_{x_{i}}\left(\eta_{i}\right)\grad_{i}T_{A}^{E}\right)^{2}\\
 & \le\sum_{\eta\in E}\mu\left(\eta\right)\,N\sum_{i=0}^{N-1}c_{x_{i}}\left(\eta_{i}\right)\,\left(\grad_{i}T_{A}^{E}\right)^{2}\\
 & \le\sum_{\eta\in E}\sum_{i=0}^{N-1}\sum_{\eta^{\prime}}\sum_{x}\,\One_{\eta^{\prime}=\eta_{i}}\One_{x=x_{i}}\frac{\mu\left(\eta\right)}{\mu\left(\eta^{\prime}\right)}\mu\left(\eta^{\prime}\right)\,Nc_{x}\left(\eta^{\prime}\right)\,\left(\grad_{x}T_{A}^{E}\left(\eta^{\prime}\right)\right)^{2}\\
 & \le Nq^{-D}\sum_{\eta^{\prime}}\sum_{x}\sum_{\eta\in E}\sum_{i=0}^{N-1}\,\One_{\eta^{\prime}=\eta_{i}}\One_{x=x_{i}}\mu\left(\eta^{\prime}\right)\,c_{x}\left(\eta^{\prime}\right)\,\left(\grad_{x}T_{A}^{E}\left(\eta^{\prime}\right)\right)^{2}\\
 & \le N^{2}\binom{V}{D}2^{D}q^{-D}\sum_{\eta^{\prime}}\sum_{x}\mu\left(\eta^{\prime}\right)\,c_{x}\left(\eta^{\prime}\right)\,\left(\grad_{x}T_{A}^{E}\left(\eta^{\prime}\right)\right)^{2}\\
 & \le N^{2}\binom{V}{D}2^{D}q^{-D-1}\mathcal{D}T_{A}^{E}=N^{2}\binom{V}{D}2^{D}q^{-D-1}\mu\left(T_{A}^{E}\One_{E}\right),
\end{align*}
where the last equality is due to \corref{dirichletequalsexpectation_EbeforeA}.
\end{proof}
\begin{prop}\label{prop:yoyoma}
Under the hypotheses of \lemref{path_to_time}
\[
\pp_{\mu}\left(\tau_{A}>N^{2}\binom{V}{D}2^{D}q^{-D-2}\right)\le2q+\delta+\sqrt{\delta},
\]
where $\delta=1-\mu\left(E\right)$.
\end{prop}

\begin{proof}
Let $t=\left(1-\sqrt{\delta}\right)N^{2}\binom{V}{D}2^{D}q^{-D-2}$.
By \lemref{path_to_time} and the Markov inequality 
\[
\pp_{\mu}\left(T_{A}^{E}\One_{E}\left(\eta\left(0\right)\right)\ge t\right)\le2q.
\]
Since $\mu\left(E\right)=1-\delta$,
\begin{align*}
\pp_{\mu}\left(T_{A}^{E}\ge t\right) & =\pp_{\mu}\left(T_{A}^{E}\ge t\text{ and }\eta\left(0\right)\in E\right)+\pp_{\mu}\left(T_{A}^{E}\ge t\text{ and }\eta\left(0\right)\notin E\right)\\
 & \le2q+\delta.
\end{align*}
On the other hand, for all $s>0$
\[
\ee_{\mu}\left(T_{s}^{E}\right)=s\mu\left(E\right)=\left(1-\delta\right)s,
\]
and since $T_{s}^{E}\le s$ we can again apply Markov's inequality
(for the positive variable $s-T_{s}^{E}$), obtaining
\[
\pp_{\mu}\left(T_{s}^{E}\le\left(1-\sqrt{\delta}\right)s\right)\le\sqrt{\delta}.
\]
In particular, chosing $s=N^{2}\binom{V}{D}2^{D}q^{-D-2}$ yields
\begin{align*}
\pp_{\mu}\left(\tau_{A}\ge s\right) & \le\pp_{\mu}\left(T_{\tau_{A}}^{E}\ge T_{s}^{E}\right)\le\pp_{\mu}\left(T_{\tau_{A}}^{E}\ge t\right)+\pp_{\mu}\left(T_{s}^{E}\le t\right)\\
 & \le2q+\delta+\sqrt{\delta}.
\end{align*}
\end{proof}
The lower bound can be obtained by comparison to the associated bootstrap
percolation (see also \cite{CMRT_kcmzoo}).
\begin{defn}
$\tau_{0}^{\text{BP}}$ is the $\mu$-random variable describing the 
infection time of the origin for bootstrap percolation.
\end{defn}

\begin{lem}
\label{lem:upperbound_bp}Fix $t>0$, $\delta>0$, and assume $\mu\left(\tau_{0}^{\text{BP}}<100t\right)<\delta$.
Then $\pp_{\mu}\left(\tau_{0}<t\right)<\delta+e^{-t}$.
\end{lem}

\begin{proof}
By the finite speed of propagation property (see, e.g., \cite[section 3.3]{Martinelli_SaintFlourNotes}), setting $X = {x : \norm{x}_1 \ge 100t}$, we may couple the dynamics starting at a configuration $\eta$ with the dynamics starting at the configuration $\eta^{X\leftarrow h}$, such that with probability at least $1-e^{-t}$ the state of the origin in both dynamics is equal up to time $t$. By the definition of the bootstrap percolation, if $\tau_0^{\text{BP}}\ge 100t$, than the dynamics starting from the state $\eta^{X\leftarrow h}$ could never infect the origin. Therefore, the dynamics starting at $\eta$ could infect the origin with probability at most $e^{-t}$. This concludes the proof.
\end{proof}

\section{KCM on polluted $\protect\zz^{2}$}

The upper bound for $\tau_0$ is given by \lemref{upperbound_bp} and the estimates
of $\tau_{0}^{\text{BP}}$ in \cite{AL} for the non-polluted case (i.e., when
all sites are susceptible), together with the observation that by
adding immune sites $\tau_{0}^{\text{BP}}$ could only increase.

For the lower bound, we start by fixing two scales: 
\begin{align}
L & =q^{-1-\varepsilon/3},\label{eq:fa2polluted_Landl}\\
l & =q^{-L-1}.\nonumber 
\end{align}
\begin{defn}
A square (that is, a subset of $\zz^{2}$ of the form $x+\left[L\right]^{2}$)
is \textit{good} if all its sites are susceptible and each row and column contain at least one infected site.
\end{defn}

\begin{claim}
\label{claim:fa2polluted_goodprob}For $q$ small enough $\nu\otimes\mu\,\left(\left[L\right]^{2}\text{ is good}\right)\ge1-2q^{\varepsilon/3}$.
\end{claim}

\begin{proof}
The probability that one of the sites of $\left[L\right]^{2}$ is
immune is at most $L^{2}\pi$, which is bounded by $q^{-2-2\varepsilon/3}\,q^{2+\varepsilon}=q^{\varepsilon/3}$.
The probability that one of the line or columns of $\left[L\right]^{2}$
is entirely healthy is at most $2L\left(1-q\right)^{L}$, which is
asymptotically equivalent to $2q^{-1-\varepsilon/3}e^{-q^{-\varepsilon/3}}$.
This bound tends to $0$ much faster than $q^{\varepsilon/3}$, and
the union bound given the proof of the claim.
\end{proof}
We will consider the coarse grained lattice, i.e., the boxes of the form $L\hat{x} + [L]^2$ for $\hat{x}\in \zz^2$. The boxes of this lattice do not overlap, thus they are good or not good independently. That is, the notion of a good box defines a Bernoulli percolation process on the coarse grained lattice.
Together with results from percolation theory (e.g. \cite[Theorem 1.33]{DCT_PercolationShaprness,Grimmett})
this implies the following corollary.
\begin{cor}
\label{cor:fa2polluted_infinitegood}The $\nu\otimes\mu$-probability
that the origin belongs to an infinite cluster of good boxes is at
least $1-16q^{\varepsilon/3}$.
\end{cor}

\begin{defn}
Consider a path of good boxes on the course-grained lattice. We say that the path is \textit{super-good}
if one of its boxes contains an infected line.
\end{defn}

\begin{claim}
\label{claim:fa2polluted_supergoodifgood}Fix a self avoiding path
of boxes whose length is $l$. Then
\[
\nu\otimes\mu\left(\text{path is super-good }|\text{ path is good}\right)\ge1-e^{-1/q}.
\]
\end{claim}

\begin{proof}
Since the events $\left\{ \text{the path is good}\right\} $ and $\left\{ \text{one of the boxes contains an infected line}\right\} $
are both increasing we can use the FKG inequality \cite{Grimmett}, and bound this
probability by the probability that a length $l$ path of boxes (not
necessarily good) does not contain an infected line. This conclude
the proof, since
\[
\left(1-q^{L}\right)^{lL}\le e^{-q^{L}\,q^{-L-1}}.
\]
\end{proof}
\begin{claim}
\label{claim:fa2polluted_probofsg}For $q$ small enough $\nu\otimes\mu\left(0\text{ belongs to a super-good path of length }l\right)\ge1-25q^{\varepsilon/3}$.
\end{claim}

\begin{proof}
By \corref{fa2polluted_infinitegood} the origin belongs to an infinite
cluster of good boxes with probability greater than $1-16q^{\varepsilon/3}$.
In particular, it is contained in a self-avoiding path of length $l$.
Then we use \claimref{fa2polluted_supergoodifgood} and the union
bound to conclude the proof.
\end{proof}
\begin{defn}
$p_{\text{SG}}\left(\omega\right)$ is the $\mu$-probability that
the origin is contained in a super-good path of length $l$.
\end{defn}

\begin{defn}
We say that $\omega$ is \textit{low pollution} if $p_{\text{SG}}\left(\omega\right)>1-5q^{\varepsilon/6}$.
\end{defn}

\begin{claim}
\label{claim:fa2polluted_proboflowpollution}$\nu\left(\text{low pollution}\right)\ge1-5q^{\varepsilon/6}$.
\end{claim}

\begin{proof}
By \claimref{fa2polluted_probofsg}, $\nu\left(p_{\text{SG}}\right)\ge1-25q^{\varepsilon/3}$.
Since $p_{\text{SG}}\le1$, Markov inequality will give the result.
\end{proof}
From now on we think of a fixed $\omega$. Let
\begin{align*}
E & =\left\{ 0\text{ belongs to a super-good path of length }l\right\} ,\\
A & =\left\{ \eta_{0}=0\right\} .
\end{align*}
\begin{prop} \label{prop:z2path}
$E,A$ satisfy the assumptions of \lemref{path_to_time} with $N=4L^{2}l,D=3L,V=3L^{2}$.
\end{prop}

\begin{proof}
The path is constructed by propagating an infected column (or row), as illustrated in Figures \ref{fig:propagationgacol}, \ref{fig:rotatingcol}, and \ref{fig:z2path}.
Figure \ref{fig:propagationgacol} shows how an infected column could propagate to the right in a good box. Since the path may have corners, we will occasionaly need to rotate the infected column and create an infected row, as explained in Figure \ref{fig:rotatingcol}. Finally, using these two basic moves, we may take the column or row that was initially infected by the assumption that the path is super-good, and then move it along the path until the origin is infected. This is illustrated in Figure \ref{fig:z2path}.
\end{proof}

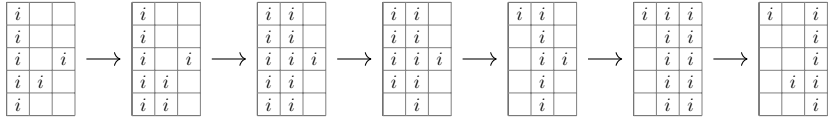
\begin{figure}
\begin{tikzpicture}[scale=0.3, every node/.style={scale=0.6}]
	\def\xx{0}
	\def\yy{0}
	
	\draw[step=5, black, very thin,xshift=\xx cm, yshift=\yy cm] (0,0) grid +(3,5);
	\draw[step=1, gray, very thin,xshift=\xx cm, yshift=\yy cm] (0,0) grid +(3,5);

	\foreach \y in {0,...,4}{
		\draw (\xx+0.5,\yy+\y+0.5) node[black] {$i$};
	}
	
	\draw (\xx+1.5,\yy+1.5) node[black] {$i$};
	\draw (\xx+2.5,\yy+2.5) node[black] {$i$};

	\draw[->]  (3.5,2.5) to (5,2.5);

	\def\xx{5.5}
	\def\yy{0}
	
	\draw[step=5, black, very thin,xshift=\xx cm, yshift=\yy cm] (0,0) grid +(3,5);
	\draw[step=1, gray, very thin,xshift=\xx cm, yshift=\yy cm] (0,0) grid +(3,5);

	\foreach \y in {0,...,4}{
		\draw (\xx+0.5,\yy+\y+0.5) node[black] {$i$};
	}

	\draw (\xx+1.5,\yy+1.5) node[black] {$i$};
	\draw (\xx+1.5,\yy+0.5) node[black] {$i$};
	\draw (\xx+2.5,\yy+2.5) node[black] {$i$};

	\draw[->]  (9,2.5) to (10.5,2.5);

	\def\xx{11}
	\def\yy{0}
	
	\draw[step=5, black, very thin,xshift=\xx cm, yshift=\yy cm] (0,0) grid +(3,5);
	\draw[step=1, gray, very thin,xshift=\xx cm, yshift=\yy cm] (0,0) grid +(3,5);

	\foreach \y in {0,...,4}{
		\draw (\xx+0.5,\yy+\y+0.5) node[black] {$i$};
	}
	\foreach \y in {0,...,4}{
		\draw (\xx+1.5,\yy+\y+0.5) node[black] {$i$};
	}

	\draw (\xx+2.5,\yy+2.5) node[black] {$i$};

	\draw[->]  (14.5,2.5) to (16,2.5);

	\def\xx{16.5}
	\def\yy{0}
	
	\draw[step=5, black, very thin,xshift=\xx cm, yshift=\yy cm] (0,0) grid +(3,5);
	\draw[step=1, gray, very thin,xshift=\xx cm, yshift=\yy cm] (0,0) grid +(3,5);

	\foreach \y in {1,...,4}{
		\draw (\xx+0.5,\yy+\y+0.5) node[black] {$i$};
	}
	\foreach \y in {0,...,4}{
		\draw (\xx+1.5,\yy+\y+0.5) node[black] {$i$};
	}
	
	\draw (\xx+2.5,\yy+2.5) node[black] {$i$};

	\draw[->]  (20,2.5) to (21.5,2.5);

	\def\xx{22}
	\def\yy{0}
	
	\draw[step=5, black, very thin,xshift=\xx cm, yshift=\yy cm] (0,0) grid +(3,5);
	\draw[step=1, gray, very thin,xshift=\xx cm, yshift=\yy cm] (0,0) grid +(3,5);

	\draw (\xx+0.5,\yy+4.5) node[black] {$i$};

	\foreach \y in {0,...,4}{
		\draw (\xx+1.5,\yy+\y+0.5) node[black] {$i$};
	}
	
	\draw (\xx+2.5,\yy+2.5) node[black] {$i$};

	\draw[->]  (25.5,2.5) to (27,2.5);

	\def\xx{27.5}
	\def\yy{0}
	
	\draw[step=5, black, very thin,xshift=\xx cm, yshift=\yy cm] (0,0) grid +(3,5);
	\draw[step=1, gray, very thin,xshift=\xx cm, yshift=\yy cm] (0,0) grid +(3,5);

	\draw (\xx+0.5,\yy+4.5) node[black] {$i$};
	
	\foreach \y in {0,...,4}{
		\draw (\xx+1.5,\yy+\y+0.5) node[black] {$i$};
	}	
	
	\foreach \y in {0,...,4}{
		\draw (\xx+2.5,\yy+\y+0.5) node[black] {$i$};
	}

	\draw[->]  (31,2.5) to (32.5,2.5);

	\def\xx{33}
	\def\yy{0}
	
	\draw[step=5, black, very thin,xshift=\xx cm, yshift=\yy cm] (0,0) grid +(3,5);
	\draw[step=1, gray, very thin,xshift=\xx cm, yshift=\yy cm] (0,0) grid +(3,5);

	\draw (\xx+0.5,\yy+4.5) node[black] {$i$};
	
	\foreach \y in {0,...,4}{
		\draw (\xx+2.5,\yy+\y+0.5) node[black] {$i$};
	}
	
	\draw (\xx+1.5,\yy+1.5) node[black] {$i$};
	
\end{tikzpicture}

\caption{\label{fig:propagationgacol}Illustration of the proof of \propref{z2path}.
We see here how an infected column could propagate in a good box. $i$ stands for infected sites. Other sites could be either infected or healthy, according to their initial state.}
\end{figure}
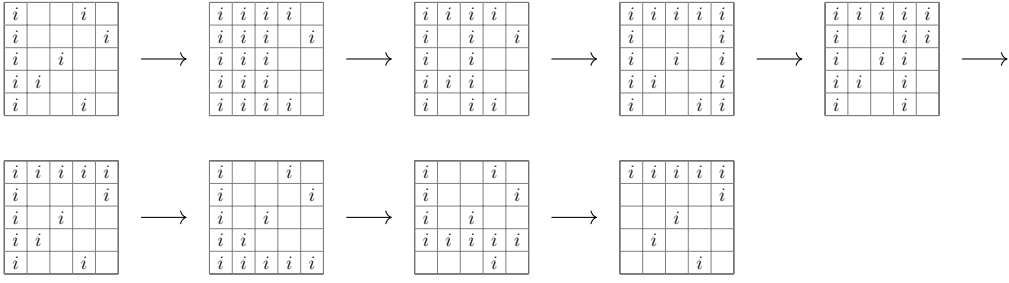
\begin{figure}
\begin{tikzpicture}[scale=0.3, every node/.style={scale=0.6}]
	\def\xx{0}
	\def\yy{0}
	
	\draw[step=5, black, very thin,xshift=\xx cm, yshift=\yy cm] (0,0) grid +(5,5);
	\draw[step=1, gray, very thin,xshift=\xx cm, yshift=\yy cm] (0,0) grid +(5,5);

	\foreach \y in {0,...,4}{
		\draw (\xx+0.5,\yy+\y+0.5) node[black] {$i$};
	}
	
	\draw (\xx+1.5,\yy+1.5) node[black] {$i$};
	\draw (\xx+2.5,\yy+2.5) node[black] {$i$};
	\draw (\xx+3.5,\yy+4.5) node[black] {$i$};
	\draw (\xx+3.5,\yy+0.5) node[black] {$i$};
	\draw (\xx+4.5,\yy+3.5) node[black] {$i$};

	\draw[->]  (6,2.5) to (8,2.5);

	\def\xx{9}
	\def\yy{0}
	
	\draw[step=5, black, very thin,xshift=\xx cm, yshift=\yy cm] (0,0) grid +(5,5);
	\draw[step=1, gray, very thin,xshift=\xx cm, yshift=\yy cm] (0,0) grid +(5,5);

	\foreach \y in {0,...,4}{
		\draw (\xx+0.5,\yy+\y+0.5) node[black] {$i$};
	}
	\foreach \y in {0,...,4}{
		\draw (\xx+1.5,\yy+\y+0.5) node[black] {$i$};
	}
	\foreach \y in {0,...,4}{
		\draw (\xx+2.5,\yy+\y+0.5) node[black] {$i$};
	}

	\draw (\xx+3.5,\yy+4.5) node[black] {$i$};
	\draw (\xx+3.5,\yy+0.5) node[black] {$i$};
	\draw (\xx+4.5,\yy+3.5) node[black] {$i$};

	\draw[->]  (15,2.5) to (17,2.5);

	\def\xx{18}
	\def\yy{0}
	
	\draw[step=5, black, very thin,xshift=\xx cm, yshift=\yy cm] (0,0) grid +(5,5);
	\draw[step=1, gray, very thin,xshift=\xx cm, yshift=\yy cm] (0,0) grid +(5,5);

	\foreach \y in {0,...,4}{
		\draw (\xx+0.5,\yy+\y+0.5) node[black] {$i$};
	}

	\foreach \y in {0,...,4}{
		\draw (\xx+2.5,\yy+\y+0.5) node[black] {$i$};
	}
	
	\draw (\xx+1.5,\yy+1.5) node[black] {$i$};
	\draw (\xx+1.5,\yy+4.5) node[black] {$i$};
	\draw (\xx+3.5,\yy+4.5) node[black] {$i$};
	\draw (\xx+3.5,\yy+0.5) node[black] {$i$};
	\draw (\xx+4.5,\yy+3.5) node[black] {$i$};

	\draw[->]  (24,2.5) to (26,2.5);

	\def\xx{27}
	\def\yy{0}
	
	\draw[step=5, black, very thin,xshift=\xx cm, yshift=\yy cm] (0,0) grid +(5,5);
	\draw[step=1, gray, very thin,xshift=\xx cm, yshift=\yy cm] (0,0) grid +(5,5);

	\foreach \y in {0,...,4}{
		\draw (\xx+0.5,\yy+\y+0.5) node[black] {$i$};
	}

	\foreach \y in {0,...,4}{
		\draw (\xx+4.5,\yy+\y+0.5) node[black] {$i$};
	}

	\draw (\xx+1.5,\yy+4.5) node[black] {$i$};
	\draw (\xx+2.5,\yy+4.5) node[black] {$i$};
	\draw (\xx+3.5,\yy+4.5) node[black] {$i$};
	
	\draw (\xx+1.5,\yy+1.5) node[black] {$i$};
	\draw (\xx+3.5,\yy+0.5) node[black] {$i$};
	\draw (\xx+2.5,\yy+2.5) node[black] {$i$};

	\draw[->]  (33,2.5) to (35,2.5);

	\def\xx{36}
	\def\yy{0}
	
	\draw[step=5, black, very thin,xshift=\xx cm, yshift=\yy cm] (0,0) grid +(5,5);
	\draw[step=1, gray, very thin,xshift=\xx cm, yshift=\yy cm] (0,0) grid +(5,5);

	\foreach \y in {0,...,4}{
		\draw (\xx+0.5,\yy+\y+0.5) node[black] {$i$};
	}

	\foreach \y in {0,...,4}{
		\draw (\xx+3.5,\yy+\y+0.5) node[black] {$i$};
	}

	\draw (\xx+1.5,\yy+4.5) node[black] {$i$};
	\draw (\xx+2.5,\yy+4.5) node[black] {$i$};
	\draw (\xx+4.5,\yy+4.5) node[black] {$i$};
	
	\draw (\xx+1.5,\yy+1.5) node[black] {$i$};
	\draw (\xx+4.5,\yy+3.5) node[black] {$i$};
	\draw (\xx+2.5,\yy+2.5) node[black] {$i$};

	\draw[->]  (42,2.5) to (44,2.5);

	\def\xx{0}
	\def\yy{-7}

	\draw[step=5, black, very thin,xshift=\xx cm, yshift=\yy cm] (0,0) grid +(5,5);
	\draw[step=1, gray, very thin,xshift=\xx cm, yshift=\yy cm] (0,0) grid +(5,5);

	\foreach \y in {0,...,4}{
		\draw (\xx+0.5,\yy+\y+0.5) node[black] {$i$};
	}
	\foreach \x in {1,...,4}{
		\draw (\xx+\x+0.5,\yy+4.5) node[black] {$i$};
	}
	
	\draw (\xx+1.5,\yy+1.5) node[black] {$i$};
	\draw (\xx+2.5,\yy+2.5) node[black] {$i$};
	\draw (\xx+3.5,\yy+0.5) node[black] {$i$};
	\draw (\xx+4.5,\yy+3.5) node[black] {$i$};

	\draw[->]  (6,-4.5) to (8,-4.5);

	\def\xx{9}
	\def\yy{-7}

	\draw[step=5, black, very thin,xshift=\xx cm, yshift=\yy cm] (0,0) grid +(5,5);
	\draw[step=1, gray, very thin,xshift=\xx cm, yshift=\yy cm] (0,0) grid +(5,5);

	\foreach \y in {0,...,4}{
		\draw (\xx+0.5,\yy+\y+0.5) node[black] {$i$};
	}
	\foreach \x in {1,...,4}{
		\draw (\xx+\x+0.5,\yy+0.5) node[black] {$i$};
	}
	
	\draw (\xx+1.5,\yy+1.5) node[black] {$i$};
	\draw (\xx+2.5,\yy+2.5) node[black] {$i$};
	\draw (\xx+3.5,\yy+4.5) node[black] {$i$};
	\draw (\xx+4.5,\yy+3.5) node[black] {$i$};

	\draw[->]  (15,-4.5) to (17,-4.5);

	\def\xx{18}
	\def\yy{-7}

	\draw[step=5, black, very thin,xshift=\xx cm, yshift=\yy cm] (0,0) grid +(5,5);
	\draw[step=1, gray, very thin,xshift=\xx cm, yshift=\yy cm] (0,0) grid +(5,5);

	\foreach \y in {1,...,4}{
		\draw (\xx+0.5,\yy+\y+0.5) node[black] {$i$};
	}
	\foreach \x in {1,...,4}{
		\draw (\xx+\x+0.5,\yy+1.5) node[black] {$i$};
	}
	
	\draw (\xx+3.5,\yy+0.5) node[black] {$i$};
	\draw (\xx+2.5,\yy+2.5) node[black] {$i$};
	\draw (\xx+3.5,\yy+4.5) node[black] {$i$};
	\draw (\xx+4.5,\yy+3.5) node[black] {$i$};

	\draw[->]  (24,-4.5) to (26,-4.5);

	\def\xx{27}
	\def\yy{-7}

	\draw[step=5, black, very thin,xshift=\xx cm, yshift=\yy cm] (0,0) grid +(5,5);
	\draw[step=1, gray, very thin,xshift=\xx cm, yshift=\yy cm] (0,0) grid +(5,5);

	\foreach \x in {0,...,4}{
		\draw (\xx+\x+0.5,\yy+4.5) node[black] {$i$};
	}

	\draw (\xx+3.5,\yy+0.5) node[black] {$i$};
	\draw (\xx+2.5,\yy+2.5) node[black] {$i$};
	\draw (\xx+1.5,\yy+1.5) node[black] {$i$};
	\draw (\xx+4.5,\yy+3.5) node[black] {$i$};
\end{tikzpicture}

\caption{\label{fig:rotatingcol}Illustration of the proof of \propref{z2path}.
We see here how to rotate an empty column in a good box.}
\end{figure}

\begin{figure}
\begin{tikzpicture}[scale=0.3, every node/.style={scale=0.5}]
	\def\xx{0}
	\def\yy{0}
	
	\draw[step=5, black, very thin,xshift=\xx cm, yshift=\yy cm] (0,0) grid +(5,5);
	\draw[step=5, black, very thin,xshift=\xx cm, yshift=\yy cm] (0,5) grid +(5,5);
	\draw[step=5, black, very thin,xshift=\xx cm, yshift=\yy cm] (5,5) grid +(5,5);
	
	\draw[step=1, gray, very thin,xshift=\xx cm, yshift=\yy cm] (0,0) grid +(5,5);
	\draw[step=1, gray, very thin,xshift=\xx cm, yshift=\yy cm] (0,5) grid +(5,5);
	\draw[step=1, gray, very thin,xshift=\xx cm, yshift=\yy cm] (5,5) grid +(5,5);
	
	\foreach \x in {0,...,4}{
		\draw (\xx+\x+0.5,\yy+0.5) node[black] {$i$};
	}

	\draw (\xx+2.5,\yy+1.5) node[black] {$i$};
	\draw (\xx+4.5,\yy+2.5) node[black] {$i$};
	\draw (\xx+3.5,\yy+3.5) node[black] {$i$};
	\draw (\xx+1.5,\yy+4.5) node[black] {$i$};

	\draw (\xx+2.5,\yy+5.5) node[black] {$i$};
	\draw (\xx+0.5,\yy+6.5) node[black] {$i$};
	\draw (\xx+1.5,\yy+7.5) node[black] {$i$};
	\draw (\xx+4.5,\yy+8.5) node[black] {$i$};
	\draw (\xx+3.5,\yy+9.5) node[black] {$i$};
	
	\draw (\xx+5.5,\yy+8.5) node[black] {$i$};
	\draw (\xx+6.5,\yy+6.5) node[black] {$i$};
	\draw (\xx+7.5,\yy+7.5) node[black] {$i$};
	\draw (\xx+8.5,\yy+5.5) node[black] {$i$};
	\draw (\xx+9.5,\yy+9.5) node[black] {$i$};
	
	\draw[->]  (11,5) to (13,5);

	\def\xx{14}
	\def\yy{0}
	
	\draw[step=5, black, very thin,xshift=\xx cm, yshift=\yy cm] (0,0) grid +(5,5);
	\draw[step=5, black, very thin,xshift=\xx cm, yshift=\yy cm] (0,5) grid +(5,5);
	\draw[step=5, black, very thin,xshift=\xx cm, yshift=\yy cm] (5,5) grid +(5,5);
	
	\draw[step=1, gray, very thin,xshift=\xx cm, yshift=\yy cm] (0,0) grid +(5,5);
	\draw[step=1, gray, very thin,xshift=\xx cm, yshift=\yy cm] (0,5) grid +(5,5);
	\draw[step=1, gray, very thin,xshift=\xx cm, yshift=\yy cm] (5,5) grid +(5,5);

	\draw (\xx+0.5,\yy+0.5) node[black] {$i$};
	
	\foreach \x in {0,...,4}{
		\draw (\xx+\x+0.5,\yy+1.5) node[black] {$i$};
	}

	\draw (\xx+4.5,\yy+2.5) node[black] {$i$};
	\draw (\xx+3.5,\yy+3.5) node[black] {$i$};
	\draw (\xx+1.5,\yy+4.5) node[black] {$i$};	

	\draw (\xx+2.5,\yy+5.5) node[black] {$i$};
	\draw (\xx+0.5,\yy+6.5) node[black] {$i$};
	\draw (\xx+1.5,\yy+7.5) node[black] {$i$};
	\draw (\xx+4.5,\yy+8.5) node[black] {$i$};
	\draw (\xx+3.5,\yy+9.5) node[black] {$i$};
	
	\draw (\xx+5.5,\yy+8.5) node[black] {$i$};
	\draw (\xx+6.5,\yy+6.5) node[black] {$i$};
	\draw (\xx+7.5,\yy+7.5) node[black] {$i$};
	\draw (\xx+8.5,\yy+5.5) node[black] {$i$};
	\draw (\xx+9.5,\yy+9.5) node[black] {$i$};

	\draw[->]  (25,5) to (27,5);

	\def\xx{28}
	\def\yy{0}

	\draw[step=5, black, very thin,xshift=\xx cm, yshift=\yy cm] (0,0) grid +(5,5);
	\draw[step=5, black, very thin,xshift=\xx cm, yshift=\yy cm] (0,5) grid +(5,5);
	\draw[step=5, black, very thin,xshift=\xx cm, yshift=\yy cm] (5,5) grid +(5,5);
	
	\draw[step=1, gray, very thin,xshift=\xx cm, yshift=\yy cm] (0,0) grid +(5,5);
	\draw[step=1, gray, very thin,xshift=\xx cm, yshift=\yy cm] (0,5) grid +(5,5);
	\draw[step=1, gray, very thin,xshift=\xx cm, yshift=\yy cm] (5,5) grid +(5,5);

	\draw (\xx+0.5,\yy+0.5) node[black] {$i$};
	\draw (\xx+2.5,\yy+1.5) node[black] {$i$};
	\draw (\xx+4.5,\yy+2.5) node[black] {$i$};
	\draw (\xx+3.5,\yy+3.5) node[black] {$i$};
	\draw (\xx+1.5,\yy+4.5) node[black] {$i$};	
	
	\foreach \x in {0,...,4}{
		\draw (\xx+\x+0.5,\yy+5.5) node[black] {$i$};
	}

	\draw (\xx+0.5,\yy+6.5) node[black] {$i$};
	\draw (\xx+1.5,\yy+7.5) node[black] {$i$};
	\draw (\xx+4.5,\yy+8.5) node[black] {$i$};
	\draw (\xx+3.5,\yy+9.5) node[black] {$i$};
	
	\draw (\xx+5.5,\yy+8.5) node[black] {$i$};
	\draw (\xx+6.5,\yy+6.5) node[black] {$i$};
	\draw (\xx+7.5,\yy+7.5) node[black] {$i$};
	\draw (\xx+8.5,\yy+5.5) node[black] {$i$};
	\draw (\xx+9.5,\yy+9.5) node[black] {$i$};

	\draw[->]  (39,5) to (41,5);

	\def\xx{4}
	\def\yy{-13}

	\draw[step=5, black, very thin,xshift=\xx cm, yshift=\yy cm] (0,0) grid +(5,5);
	\draw[step=5, black, very thin,xshift=\xx cm, yshift=\yy cm] (0,5) grid +(5,5);
	\draw[step=5, black, very thin,xshift=\xx cm, yshift=\yy cm] (5,5) grid +(5,5);
	
	\draw[step=1, gray, very thin,xshift=\xx cm, yshift=\yy cm] (0,0) grid +(5,5);
	\draw[step=1, gray, very thin,xshift=\xx cm, yshift=\yy cm] (0,5) grid +(5,5);
	\draw[step=1, gray, very thin,xshift=\xx cm, yshift=\yy cm] (5,5) grid +(5,5);

	\draw (\xx+0.5,\yy+0.5) node[black] {$i$};
	\draw (\xx+2.5,\yy+1.5) node[black] {$i$};
	\draw (\xx+4.5,\yy+2.5) node[black] {$i$};
	\draw (\xx+3.5,\yy+3.5) node[black] {$i$};
	\draw (\xx+1.5,\yy+4.5) node[black] {$i$};	
	
	\foreach \y in {0,...,4}{
		\draw (\xx+4.5,\yy+\y+5.5) node[black] {$i$};
	}

	\draw (\xx+2.5,\yy+5.5) node[black] {$i$};
	\draw (\xx+0.5,\yy+6.5) node[black] {$i$};
	\draw (\xx+1.5,\yy+7.5) node[black] {$i$};
	\draw (\xx+3.5,\yy+9.5) node[black] {$i$};
	
	\draw (\xx+5.5,\yy+8.5) node[black] {$i$};
	\draw (\xx+6.5,\yy+6.5) node[black] {$i$};
	\draw (\xx+7.5,\yy+7.5) node[black] {$i$};
	\draw (\xx+8.5,\yy+5.5) node[black] {$i$};
	\draw (\xx+9.5,\yy+9.5) node[black] {$i$};

	\draw[->]  (15,-8) to (17,-8);

	\def\xx{18}
	\def\yy{-13}

	\draw[step=5, black, very thin,xshift=\xx cm, yshift=\yy cm] (0,0) grid +(5,5);
	\draw[step=5, black, very thin,xshift=\xx cm, yshift=\yy cm] (0,5) grid +(5,5);
	\draw[step=5, black, very thin,xshift=\xx cm, yshift=\yy cm] (5,5) grid +(5,5);
	
	\draw[step=1, gray, very thin,xshift=\xx cm, yshift=\yy cm] (0,0) grid +(5,5);
	\draw[step=1, gray, very thin,xshift=\xx cm, yshift=\yy cm] (0,5) grid +(5,5);
	\draw[step=1, gray, very thin,xshift=\xx cm, yshift=\yy cm] (5,5) grid +(5,5);

	\draw (\xx+0.5,\yy+0.5) node[black] {$i$};
	\draw (\xx+2.5,\yy+1.5) node[black] {$i$};
	\draw (\xx+4.5,\yy+2.5) node[black] {$i$};
	\draw (\xx+3.5,\yy+3.5) node[black] {$i$};
	\draw (\xx+1.5,\yy+4.5) node[black] {$i$};	

	\draw (\xx+2.5,\yy+5.5) node[black] {$i$};
	\draw (\xx+0.5,\yy+6.5) node[black] {$i$};
	\draw (\xx+1.5,\yy+7.5) node[black] {$i$};
	\draw (\xx+3.5,\yy+9.5) node[black] {$i$};
	\draw (\xx+4.5,\yy+8.5) node[black] {$i$};
	
	\draw (\xx+5.5,\yy+8.5) node[black] {$i$};
	\draw (\xx+6.5,\yy+6.5) node[black] {$i$};
	\draw (\xx+7.5,\yy+7.5) node[black] {$i$};
	\draw (\xx+8.5,\yy+5.5) node[black] {$i$};
	
	\foreach \y in {0,...,4}{
		\draw (\xx+9.5,\yy+\y+5.5) node[black] {$i$};
	}
	
\end{tikzpicture}
\caption{\label{fig:z2path}Illustration of the proof of \propref{z2path}.
We see here how to propagate infection through a good path.}
\end{figure}
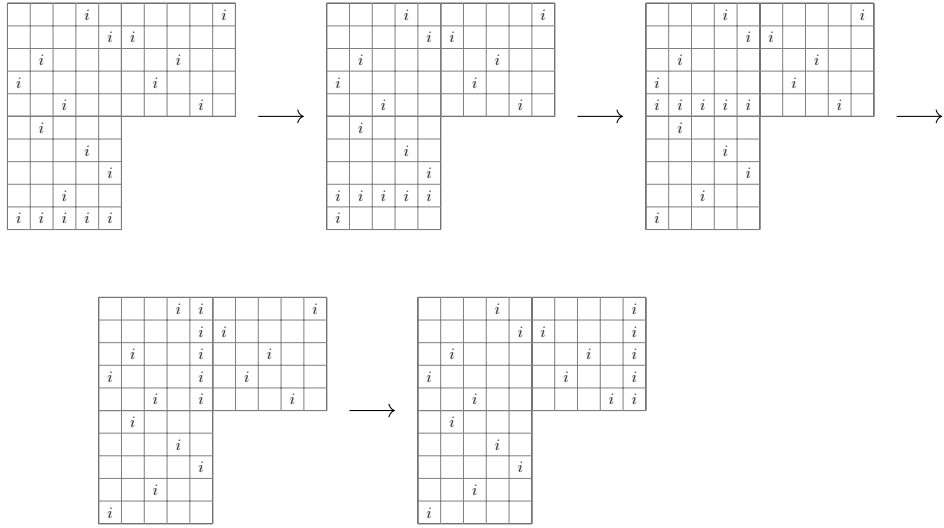

This concludes the proof of the upper bound of $\tau_0$.  For $q$ small enough $N^{2}\binom{V}{D}2^{D}q^{-D-2}\le e^{-q^{-1-\varepsilon}}$,
and by \propref{yoyoma}
\[
\pp_{\mu}\left(\tau_{A}>e^{-q^{-1-\varepsilon}}\right)\le2q+p_{\text{SG}}\left(\omega\right)+\sqrt{p_{\text{SG}}\left(\omega\right)}.
\]
In particular, if $\omega$ is low pollution $\pp_{\mu}\left(\tau_{A}>e^{-q^{-1-\varepsilon}}\right)\le3q^{\varepsilon/12}$.

\section{KCM on polluted $\zz^{3}$}

The upper bound, just as the two dimensional case, is given by \lemref{upperbound_bp}
and the results of \cite{CerfManzo}.

For the lower bound, we will use the ideas of \cite{GravnerHolroyd_polluted} in order to construct
a path that empties the origin and satisfies the hypotheses of \lemref{path_to_time}. In contrast to the infection mechanism in the bootstrap percolation (\cite{GravnerHolroyd_polluted}), we will need a finer path, that not only allows us to empty the origin, but also keeps the configuration close to the original one in order to avoid large energy barrier and entropic price.

A move is denote by a pair $(x,a)$ where $x\in \susc$ and $a\in \{i,h\}$.  A move is legal in $\eta$ if $c_x(\eta) = 1$.  A finite sequence of moves is given by $\Gamma = (x_i,a_i)_{i=1}^k$.  Starting from the configuration $\eta$, we let $\Gamma_i(\eta)$ denote the configuration obtained by applying the first $i$ moves in $\Gamma$.  The sequence $\Gamma$ is legal with respect to the initial configuration $\eta$ if each move from $\Gamma_i(\eta)$ to $\Gamma_{i+1}(\eta)$ is legal.  We let $\Gamma(\eta)$ denote the configuration obtained by applying the entire sequence of moves in $\Gamma$ starting from $\eta$. 

Note that if the change from $\eta$ to $\eta'$ is a legal move, then the change from $\eta'$ to $\eta$ is also a legal move.  Similarly, if $\Gamma$ is a legal sequence starting from $\eta$ and ending at $\eta'$ then the reverse of $\Gamma$, denoted $\rev(\Gamma)$, is a legal sequence starting from $\eta'$ and ending at $\eta.$  

We denote the concatenation of two sequences, $\Gamma$ and $\Gamma'$ by $\Gamma  + \Gamma'$.  If $\Gamma$ is a legal sequence of moves from $\eta$ to $\eta'$, and $\Gamma'$ is a legal sequence of moves from $\eta'$ to $\eta''$ then $\Gamma + \Gamma'$ is legal sequence of moves from $\eta$ to $\eta''$.

\begin{lem}\label{backandforth}
Let $\Gamma$ denote a sequence of legal moves that starts at $\eta$ and ends at $\eta'$.  For any set of susceptible sites $X$, Let $\Gamma^{X}$ denote the same set of moves as $\Gamma$ except that any move that would cure a site in $X$ is ignored.  Then $\Gamma^X$ is also a legal set of moves that start at $\eta$ and ends at $\eta'' = \Gamma^X_k(\eta)$ where for all $i\leq k$ the only difference between $\Gamma_i(\eta)$ and $\Gamma^X_i(\eta)$ lies on the set $X$.
  \end{lem}
  
\begin{proof}

This is a consequence of the fact that more infected sites can only help the constraints to be satisfied.  Let $Y_i$ denote the subset of $X$ that is infected for some $j\leq i$ in $\Gamma_i(\eta)$.    
For each $i$, $\Gamma^X_i(\eta) = (\Gamma_i(\eta))^{Y_i\leftarrow i}$, thus the infected sites in $\Gamma_i(\eta)$ are a subset of the infected sites in $\Gamma_i^X(\eta)$.  Thus moves in $\Gamma^X$ are legal if they are legal in $\Gamma.$  Finally $\eta'' = (\eta')^{Y_k\leftarrow i} \in (\eta')^{(X)}$ and thus the last line in the lemma holds.
\end{proof}

We repeat the definitions from \cite{GravnerHolroyd_polluted} of the objects necessary for our work.  Note that some of the details in the definitions are slightly modified to fit within the framework of our proof. 

For $k\in \Z$ we let $\Lambda_k$ denote $\Z^{2} \times \{k\}$ and refer to this as the $k$th layer of $\Z^3.$  We let $e_1$, $e_2,$ and $e_3$ denote the unit vectors in each of the cardinal directions.   

The \emph{standard brick} is a the collection of sites $[0,4L)\times[0,16L)\times[0,32L).$  The \emph{base} of the standard brick is the bottom half $[0,4L)\times [0,16L) \times [0,16L).$  The \emph{top} of the standard brick is the remaining half of the brick.  The \emph{sections} of the standard brick are the sets $[0,4L)\times [0,16L) \times [4jL, 4(j+1)L)$ for $0\leq j \leq 7$.  The \emph{tip} of the standard brick is the set $[0,4L) \times [0,4L), \times [16L,32L)$, and has the same dimensions as a section, though a different orientation.  The \emph{anchor} of the standard brick is the site $(4L,16L,0)$ and the \emph{flag} is the site $(0,0,32L).$  Note that the standard brick contains neither its anchor nor flag, though the flag is on the boundary of the tip and the anchor is corner opposite to the flag on the boundary of the brick. 
   
A proto-brick, $\hat{B},$ is the set of vertices $[0,4L) \times [0,4L) \times [0,2L)$ that lies in a different $\Z^3$ from the standard brick.  For each $\hat{x}=(\hat{x}_1,\hat{x}_2,\hat{x}_3)\in \hat{B}$ define $cell(x) = (x_1,4x_2,16x_3) + \{0\}\times [0,4) \times [0,16).$  The standard brick is connected to the proto-brick by $$B = \bigcup_{\hat{x}\in \hat{B}} cell(\hat{x}).$$ 
A vertex is susceptible if every site in the corresponding cell is susceptible.  

\begin{defn}

Let $\hat{B}$ be a proto-brick with corresponding brick $B$ in standard position.  The brick $B$ is good if there exists a set $\hat{S}\subseteq \hat{B}$ with the following properties
	
	\begin{enumerate}
		\item All vertices in the following set are susceptible:
$$ \sigma(\hat{S}):= \{ \hat{x}, \hat{x}+ \hat{e}_3, \hat{x}-\hat{e}_1-\hat{e}_2+\hat{e}_3: \hat{x}\in \hat{S}\}\cap \hat{B};$$
\item for all $\hat{x} \in \hat{S}$, except for $\hat{x}$ in the bottom layer of $\hat{B},$ satisfies either 
 $$ \hat{x} - \hat{e}_3 \in \hat{S}, \text{ or }$$   $$\hat{x} +\hat{e}_1+\hat{e}_2 -\hat{e}_3 \in \hat{S};$$  
\item $\hat{S}\subseteq \{ \hat{x}: 3L < \hat{x}_1 + \hat{x}_2 + \hat{x}_3 < 4L\};$
\item for $0\leq k \leq 2L-1$, $\hat{S}\cap \Lambda_k$ is an oriented path from $\{\hat{x}_1=0\}$ to $\{\hat{x}_2=0\}$ with steps $\hat{e}_1$ or $-\hat{e}_2$ where no three consecutive steps are the same; 
\item The site $(L+1, 4L+4, 16L)S$ is contained in $S$, where $$S = \bigcup_{\hat{x}\in \hat{S}} cell(\hat{x}).$$
\item for $0\leq k < 32L-1$, there is an infected site in $S\cap\Lambda_k + e_3.$  
	\end{enumerate}

\end{defn}
  
The first five conditions for a good brick depend only on the initial random set $\susc$.  The last condition also depends on the configuration on $\susc$.

The set $S$ is called the sail of $B$.  The set $$\bar{S} = S \cup \{S+ e_3\} \cap B$$ is called the thick sail of $B$.   

For a fixed $\hat{x}\in \hat{B}$ and some value $0\leq j \leq 15$ we call the set $(x_1,x_2,x_3) + \{0\} \times [0,4) \times \{j\}$ the $j$th \emph{unit} of $cell(\hat{x})$.  Similarly, for fixed $0 \leq j \leq 3$, the set $(x_1,x_2,x_3) + \{0\}\times \{j\} \times [0,16)$ is called a \emph{strip}.  A cell consists of 16 units or 4 strips.

The set $B_i = \Lambda_i \cap B$ is the $i$th \emph{layer} of $B$.  The sites in a unit of a cell all lie on the same layer, whereas the sites of a strip in a cell lie across 16 different layers.  

If each of the units in a layer is the 15th unit in a cell, we say the layer is a \emph{transition} layer.  Otherwise it is an internal layer.  Note the only sites in $\bar{S}\backslash S$ are (possibly) those that lie directly above a transition layer.

We denote the layers of a sail as the sets $S_i = \Lambda_i \cap S.$  If $S_i$ is in an internal layer, then $S_{i+1} = S_i + e_3$.  The layer $S_i$ is an oriented path of units with subsequent units differing by either $e_1$ or $-4e_2$.  The path has two types of corner units:  an exterior corner unit reached by $e_1$ from the previous unit and followed by $-4e_2$ to the next unit, and an interior corner unit reached by $-4e_2$ and followed by $e_1$.  For $i< 32L-1$, if $S_i$ is a transition layer then for every $u \in S_i$ either $u+e_3 \in S_{i+1}$ or $u+e_3-e_1-4e_2 \in S_{i+1}.$   

A general brick, $B$, is given by an anchor and flag in $\Z^3$ that differ by $(d_1,d_2,d_3)$ given by some rearrangement of the vector $(4L,16L,32L)$.  The base, top, sections, layers, and tip of a general brick is the reorientation of the those of the standard brick to fit with the new choice of anchor and flag. The definitions of sails, units, strips, etc. are also all modified according to the new orientation.

In \cite{GravnerHolroyd_polluted} they essentially show that a good sail can become completely infected under the bootstrap percolation dynamics if the bottom layer is completely infected.  We need a refinement of this result, as we do not wish to infect everything, but instead we wish to move the infection from some set to another in $\Z^3$ whilst not increasing the number of infected sites of the configurations restricted to a finite domain by too much.  

The following series of lemmas will show how to propagate infection from certain sets of sites in the base of $B$ to any collection of sites in the tip of $B$.  

\begin{lem}\label{linebyline}

Let $\eta$ be an initial configuration on the brick $B$.  Suppose $B$ is good with respect to $\eta$ and let $S$ denote the sail of $B$.  Fix $i$ and consider the configuration $\eta' = \eta^{S_i \leftarrow i }$.  Then there exists a sequence of legal moves from $\eta'$ to $\eta''$ of length at most $CL$ where every site in $S_{i+1}$ becomes infected in $\eta''$, and $\eta''$ agrees with $\eta$ outside of $S_{i+1}$ and possibly a set of sites consisting of at most 2 boundary sites in $S_i$ and 16 sites in $S_i+e_3$ that lie on the boundary of $B$.
	
\end{lem}

\begin{proof}

Without loss of generality assume $B$ is the standard brick.  By property $5$ of a good brick there exists some infected site in $S_{i}+e_3$.  We proceed unit by unit. Suppose $u$ is a unit in $S_{i} + e_3$ that contains an infected site, $z$.  First we may cure $z-e_3\in S_i.$  Then we may infected a site in $S_i+e_3$ connected to the already emptied sites of $S_i+e_3$ and cure the corresponding site in $S_i$ below the newly emptied site.  The set $S_i + e_3$ is connected so we may proceed until every site in $S_i + e_3$ is infected and all but one site in each of the boundary units of $S_i$ are healthy.  

If $S_i$ is an internal layer then $S_{i+1} = S_i + e_3$ and we are done.  Otherwise suppose $S_i$ is a transition layer and let $P_0 = S_i + e_3$ be viewed as a path of units.  For $l\geq 0$ we proceed inductively.  Suppose $P_l$ is a path of infected units consisting of steps $e_1$ and $-4e_2$ such that for every $u\in P_l$, either $u\in S_{i+1}$ or $u-e_1-4e_2\in S_{i+1}$.  If $P_l$ contains $S_{i+1}$ then it differs from $S{i+1}$ by at most four boundary units and we are done.  Otherwise there exists some non-boundary external corner unit $u \in P_l$ such that $u\notin S_{i+1}$ and $u-e_1-4e_2 \in P_l$.  Since $u$ is an external corner unit, $u-e_1$ and $u-4e_2\in P_l$.  Thus we infected every site in $u-e_1-4e_2$ and cure $u$.  This creates a new path of units $P_{l+1}$ that satisfies either $u'\in S_{i+1}$ or $u'-e_1-4e_2 \in S_{i+1}$ for every $u'\in P_{l+1}.$  Continue until no external corner units exists.  This final path of infected units will contain all of $S_{i+1}$ plus at most 4 boundary units of $S_{i} +e_3$.   
	
\end{proof}

The following lemma generalizes the statement of Lemma \ref{linebyline}.  

\begin{lem}\label{separator}
Let $\eta$ be an initial configuration on the brick $B$.  Suppose $B$ is good with respect to $\eta$ and let $S$ denote the sail of $B$.  Let $V$ be a set of sites restricted to a single section in the base of $B$, such that $V$ separates $B$ into two connected components $B^-$ and $B^+$, where $B^+$ is the part containing the top of $B$.  Let $X_1 = V\cap \bar{S}$ and $X_2$ be any subset of $B^+\cap\bar{S}$.  There is a sequence of legal moves of length $O(L^2)$, starting from $\eta^{X_1\leftarrow i}$ and ending at $\eta^{ X_1\cup X_2 \leftarrow i }$, where at any time in the sequence the configuration differs from $\eta$ on a set of at most $|X_1| + |X_2| + O(L)$ sites.

\end{lem}

\begin{proof}

Without loss of generality assume $B$ is the standard brick.  Any results here apply by changing the orientation of the moves.  

Since $V$ is a separating set for $B$, The paths (of units), $S_j$ and $S_j+e_3$, are partitioned in to connected collections of sites are either entirely in $B^+$ or $B^-$.

Let $i$ denote the lowest layer in $B$ that contains a site in $V\cap \bar{S}$.  No site in $S_i$ is in $B^+$, otherwise $V$ would not separate $B$.  Thus each $x\in S_i$ is either in $V$ or $B^-$.  

Consider the initial configuration $\eta'=\eta^{S_i\leftarrow i}$ obtained by infecting every site in $S_i$ from $\eta$.  By Lemma \ref{linebyline} there exists a sequence of legal moves starting from $\eta'$ that infects every site in $S_{i+1}$ while leaving healthy all but 20 sites in $S_i$ and $S_{i} + e_3$.  Repeat this process, infecting subsequent lines, except for the boundary points.  Any time a site in $X_2$ is to become healthy, leave it infected.  Eventually every site in $X_2$ will become infected as this process will at some point infect every site in $S$ that lies above $S_i$, which includes all of $S\cap B^+$.  Infecting each subsequent line and curing the previous takes at most $CL$ steps and leaves behind at most 20 infected sites on each line that are not necessarily in $X_2$.  There are at most $32L$ layers that need to be infected in $\bar{S}\cap B$ before every site in $X_2$ has been infected.  The most number of infected sites in this process is $|\eta| +  |S_i| + |X_2| + O( L).$  Let $\Gamma'$ denote this sequence of legal moves that starts at $\eta'$ and ends with $X_2$ completely infected.  The number of moves in $\Gamma'$ is at most $O(L^2)$.

Each step in $\Gamma'$ consists of curing or infecting a site $x$, or leaving it unchanged.  Now consider the related censored sequence of moves, $\Gamma$, starting from $\eta'=\eta^{X_1\leftarrow i}$ defined as follows. Suppose $x$ is the site to be updated in $\Gamma'$ at time $t$.  If $x\in B^-$ or $x\in X_1$ then do nothing in $\Gamma$ at time $t$.  Otherwise act in the same way as $\Gamma'$.  We claim that for each $t$ the set of infected sites in $\Gamma'_t(\eta')\cap B^+$ is the same as the set of infected sites in $\Gamma_t(\eta^{S_i\leftarrow i})\cap B^+.$  Suppose for $t\geq 0$ the claim is true up to time $t$, and the move at time $t$ is at a site $z_t\in \bar{S}$.  In $\Gamma'_{t}(\eta')$, $z_t$ is a legal move, so there exists at least $2$ neighbors of $z_t$ that are in $\bar{S}$.  Those neighbors are either in $B^+$ or $X_1$.  If they are in $B^+$, then by the claim they must be in $\Gamma_t(\eta)$ and thus infected.  Otherwise they are in $X_1$ in which case they also must be infected in $\Gamma_t(\eta).$  In either case the move at $z_t$ is legal and therefore $\Gamma_{t+1}(\eta)(z_t) = \Gamma'_{t+1}(\eta')(z_t)$, and the claim remains true for $t+1$.

Thus every move in $\Gamma$ is legal.  Moreover, $X_2\subseteq B^+$ and thus $X_2$ is infected in $\Gamma_t(\eta^{X_1\leftarrow i})$ when it is infected in $\Gamma'_t(\eta').$  The configuration $\Gamma_t(\eta)$ differs from $\eta$ on a set of size at most $|X_1|+ |X_2| + O(L)$ and the length of $\Gamma$ equals the length of $\Gamma'$.

Once $X_2$ is completely infected, reverse the sequence of moves, except anytime a site in $X_1$ or $X_2$ is to be cured, ignore that move, thus leaving that site infected.  The resulting configuration will have sites that are infected only if they are either in $X_1$ or $X_2$ or were infected in $\eta.$  The full sequence will take at most twice the number of steps in $\Gamma$, $O( L^2)$.   
\end{proof}

The following corollary shows how to cure a collection of infected sites that lie further up the sail.  

\begin{cor}\label{mrclean}
Let $B$ be a good brick with sail $S$ for a configuration $\eta$ on $B$.  Let $V$ be a set of sites that separates $B$ into two connected components $B^+$ and $B^-$.  Let $X_2$ be a set in $B^+\cap \bar{S}$ and let $X_1 = V \cap \bar{S}$.    There exists a sequence of moves of length at most $CL^2$ that begins at $\eta^{X_1\cup X_2\leftarrow i}$ and ends at $\eta^{X_1 \leftarrow i}$ where each configuration in the sequence differs from $\eta$ on a set of size at most $ |X_1| + |X_2| + O(L).$  
\end{cor}

\begin{proof}

By Lemma \ref{separator}, there is a sequence of moves that starting from $\eta^{X_1\leftarrow i}$ and ending at $\eta^{X_1\cup X_2 \leftarrow i}.$  Apply the reverse of this sequence.
\end{proof}

Starting from a brick $B$ with some layer in the base of the sail completely infected, we will show how to propagate this infection to a translation of $B$, while curing most of the infected sites in the original brick $B$.


A brick, $B$, \emph{points} to another brick, $B'$, if the tip of $B$ coincides exactly with one of the four sections in the base of $B'$.  This is denoted by $B \triangleright_k B'$ where $k$ is the corresponding section in $B'$ which coincides with the tip of $B$.  


\begin{lem}\label{sail_intersect}
Let $B$ be a good brick that points to a good brick $B'$, with sails $S$ and $S'$ respectively.  Then $|S\cap S'| = O(L)$.
	
\end{lem}

\begin{proof}
	Assume without loss of generality the cells of $B$ have orientation given by $(1,4,16)$ while the cells of $B'$ have orientation $(16,1,4)$.  
	
	Fix a layer in $B_i \subset B$ in the tip of $B$.  By definition $X_i = B_i \cap S$ is a path of units, denote $\{u_i\}$ with steps $u_{i+1}-u_i \in \{e_1,4e_2\}$ such that at most three consecutive steps are equal.  
	
Similarly the set $X'_i = S' \cap B_i$ is a path of strips $\{s_i\}$.  Without loss of generality, the step $s_{i+1}-s_i \in \{e_2,16e_1,e_2+16e_1\}$ with at most three consecutive steps of type $e_2$.  

Let us consider the $e_1$ coordinates of $\{u_i\}$ and $\{s_i\}$.  Every four steps the $e_1$ coordinate of $\{u_i\}$ grows by at most 3 while the $e_1$ coordinates of $\{s_i\}$ grows by at least $16$.  Thus, the size of the intersection of $X_i$ and $X'_i$ is $O(1)$ and therefore the size of intersection of $S$ and $S'$ is $O(L)$.  
\end{proof}

The next lemma is an extension of Lemma \ref{separator} which allows us to pass infection from the base of one brick to the tip of the next while staying within an $O(L)$ Hamming distance from the original configuration.

\begin{lem} \label{passiton}

Let $\eta$ be a configuration on $B\cup B'$ such that $B$ and $B'$ are good bricks with respect to $\eta$ such that $B$ points to $B'$ with corresponding oriented sails $S$ and $S'$.  Let $V$ be a separating set in a section of the base of $B$, and $X_1 = V\cap \bar{S}$.  For any subset of sites $X'$ in the tip of $B'\cap \bar{S}'$, there exists a sequence of $O(L^2)$ moves starting from $\eta^{X_1\leftarrow i}$ and ending at $\eta^{X_1 \cup X'\leftarrow i}$ while each configuration in the sequence differs from $\eta$ on a set of size at most $|X_1| + |X_2| + |X'|+O(L)$.
	
\end{lem}

\begin{proof}

Let $X_2 = S\cap \bar{S}'$.  From Lemma \ref{separator} there exists a sequence of length at most $O(L^2)$ of moves $\Gamma$ starting at $\eta^{X_1\leftarrow i}$ and ending at $\eta^{X_1\cup X_2 \leftarrow i}$ such that the number of infected sites at any step in the sequence is at most $|I(\eta)| +|X_1| + |X_2|+O(L)$.  By Lemma \ref{sail_intersect}, the size of $S\cap S'$ is at most $O(L)$, and the size of $X_2$ is at most double the size of $S\cap S'$.  Moreover, $S$ is a separating set for $B'$ and thus, again by Lemma \ref{separator}, there exists a sequence of moves, $\Gamma'$ starting from $\eta^{X_1\cup X_2\leftarrow i}$ that ends in $\eta^{X_1\cup X_2\cup  X'\leftarrow i}$ within Hamming distance $|X_1| + |X_2| + |X'| + O(L)$ if $\eta$.  Finally, the sequence $\rev(\Gamma)$ is a legal sequence of moves starting from $\eta^{X_1\cup X_2\cup X'\leftarrow i}$ and ending at $\eta^{X_1\cup X' \leftarrow i}$ since $\eta^{X_1\cup X_2\leftarrow i}$ agrees with $\eta^{X_1\cup X_2\cup X'\leftarrow i}$ on $B$ and the moves of $\Gamma$ are contained entirely in $B$.  The sequence of moves $\Gamma + \Gamma' + \rev(\Gamma)$ satisfies conditions of the lemma and has The length of these three sequence is at most $O(L^2)$.  	
\end{proof}

We specify a sequence of bricks $\trans = \{B^i\}_{i=0}^3$.

\begin{itemize}
\item $B^0 = \left( (0,0,0); (4L,16L,32L) \right)$,
\item $B^1 = \left( (-12L,16L,32L); (20L,12L,16L) \right)$,
\item $B^2 = \left((20L,0,16L);(4L,32L,20L)\right)$,
\item $B^3 = \left((4L,16L,4L);(8L,32L,36L)\right)$.
\end{itemize}

Similarly define the sequence $\trans'= \{{A}^i\}_{i=0}^3$ as:

\begin{itemize}
\item $A^0 = B^0$,
\item $A^1 = B^1$,
\item $A^2 = \left( (4L,0L,16L);(20L,32L,20L)\right)$,
\item $A^3 = \left((16L,16L,4L);(20L,32L,36L) \right)$.
\end{itemize}
 
The sequences satisfy 
$$B^0\points_3 B^1\points_3 B^2 \points_3 B^3$$
and
$$A^0\points_3 A^1\points_3 A^2 \points_3 A^3.$$
and both $B^3$ and $A^3$ are translations of $B^0 = A^0$ See Fig. \ref{bricktranslate}.

\begin{figure}\label{bricktranslate}
\includegraphics[scale=.5]{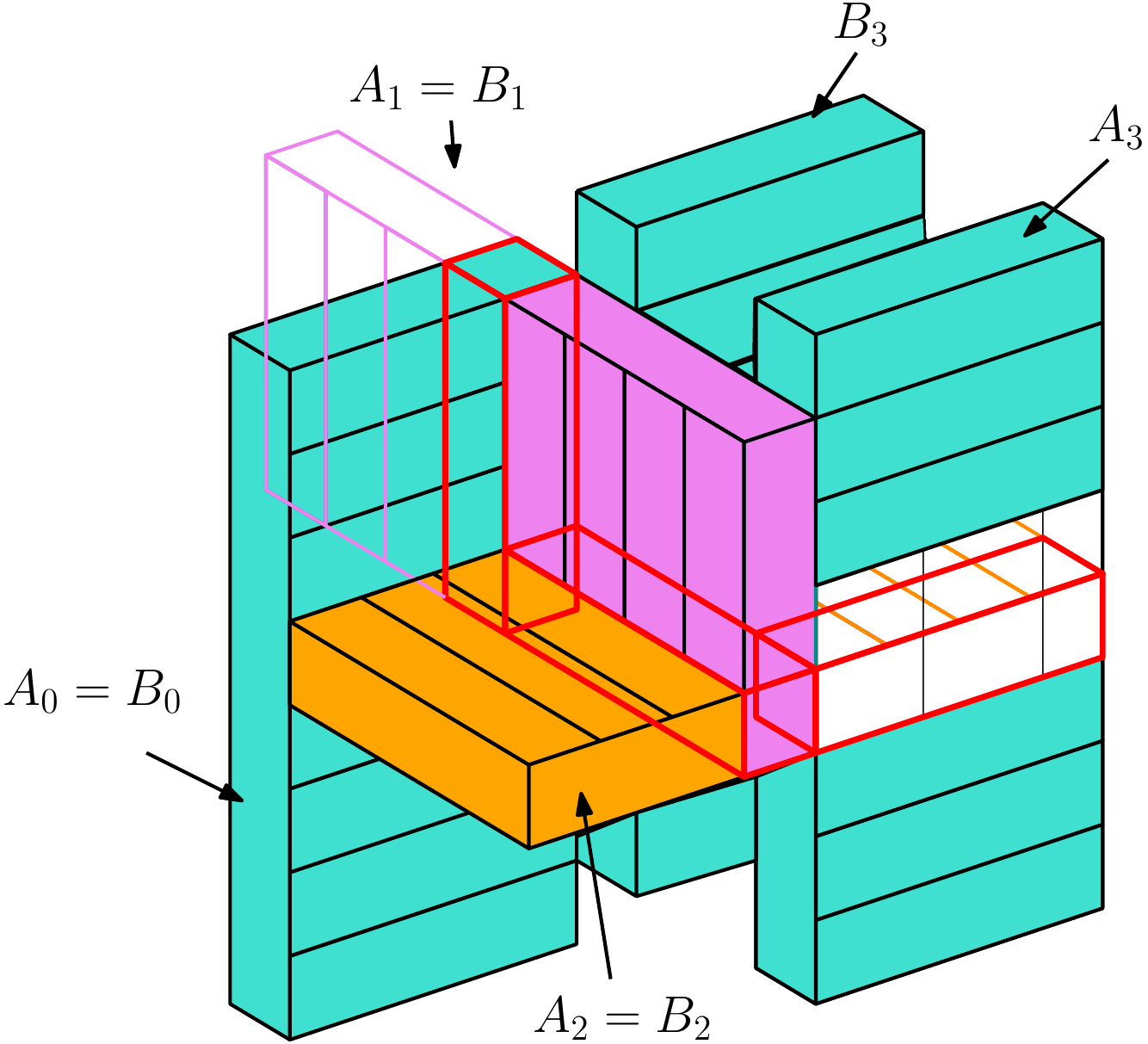}	
\caption{$B^0= A^0$ is the left turquoise brick, $B^1$ and $A^1$ are represented by the pink brick, $B^2$ and $A^2$ are both represented by the orange brick (with different sail orientations), and $B^3$ and $A^3$ are the turquoise translations to the right of $B^0$.  The red outlines are where the tip of a brick coincides with the base othe another brick.}
\end{figure}

Note that $B^2$ and $A^2$ share the same set of sites but have flipped orientations.  Both $\trans$ and $\trans'$ are called translation sequences.  Given certain conditions on the bricks in $\trans$ or $\trans'$, certain infected sites in $B$ can propagate to infected sites in $B^3$ or $A^3$. 

We also define another sequence of bricks called a cleaning sequence, denoted $\clean = \{C^i\}_{i=0}^7$.  This sequence of bricks is used to return the sites in $\trans$ or $\trans'$ outside $B^3$ or $A^3$ to some original state.  We will not specify the cleaning sequence as it is not unique.  Figure \ref{fig:cleaning} gives and example.

\begin{figure}\label{fig:cleaning}
\includegraphics[scale=.5]{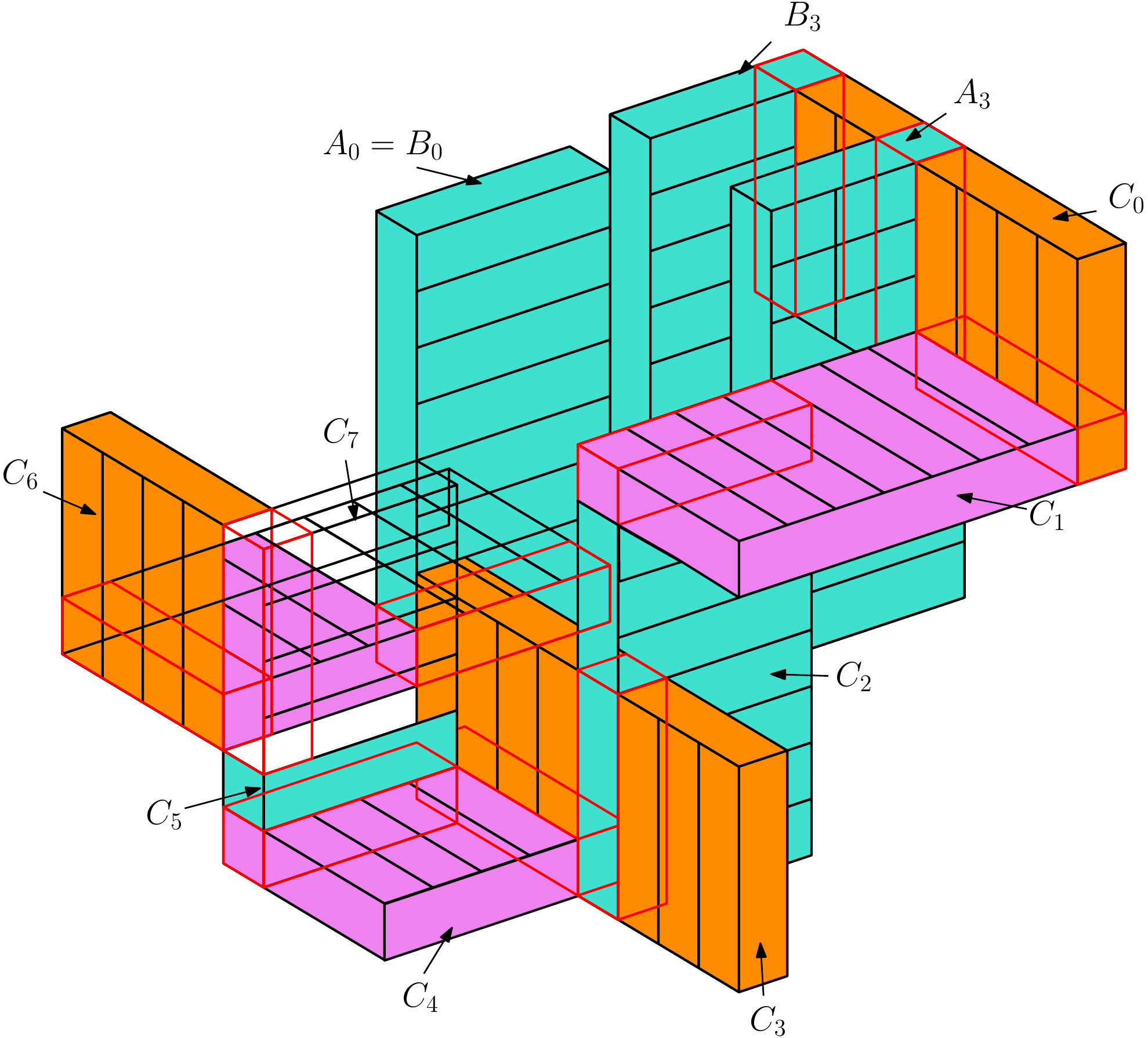}
\caption{A sequence of eight bricks $\{C^i\}$, that facilitate cleaning outside $A^3$ or $B^3$.  The tips of $B^3$ and $A^3$ point to the bottom section and $3$rd section, respectively, of $C^0$.  Starting from $B^3$ this example satisfies, $B^3 \points_0 C^0\points_0 C^1\points_0 C^2\points_3 C^3\points_3 C^4 \points_2 C^5 \points_3 C^6\points_0 C^7 \points_0 B^0.$  The colors correspond to orientation.  Some unused sections of bricks were removed or made transparent to show hidden bricks.  Intersections of tips with sections in the base of subsequent bricks are highlighted in red.}	
\end{figure}

%

Our goal will be to show that we can infect sites in either $B^3$ or $A^3$ given a specific infected set of sites $X$ in section $3$ of $B = B^0 = A^0$ and conditions on the bricks $B^i$ or $A^i$ for $i=0,1,2,3$.  Moreover, once we have spread the infected set of sites from $B$ to either $B^3$ or $A^3$ we can return the configuration outside of $B^3$ or $A^3$ back to its original state but with the sites at $X$ no longer infected.

\begin{lem} \label{allinfamily}

Let $\eta$ be a configuration on the sites of the bricks in $\trans\cup \clean.$  Let the bricks in $\trans$ and $\clean$ be good and have sails pointing in the appropriate direction, denoted by $S^i$ and $R^i,$ respectively, for the sails of the corresponding brick in $\trans$ or $\clean$.

Let $V$ be some a separating set in section 3 of $B^0$ and $X = V\cap \bar{S}^0$.  Let $X' = S^2 \cap \bar{S}^3$.  There exists a sequence of legal moves, $\Gamma$, starting from $\eta^{X\leftarrow i}$ and ending at $\eta^{X' \leftarrow i}$.  The length of $\Gamma$ is at most $O(L^2)$ and each configuration in the sequence differs from $\eta$ on a set of size $O(L).$  A similar statement holds for $\trans'$.     

\end{lem}

\begin{proof}

This is simply an extension of Lemma \ref{passiton}.  For $i=0,1,2$ each of the sails, $S^i$, are separating sets in the section 3 of the subsequent brick $B^{i+1}$.  For $i=0,1,2$, let $X^i = S^i\cap \bar{S}^{i+1}$ and $\Gamma^i$ be a sequence of legal moves from $\eta^{X^i\leftarrow i}$ that ends at $\eta^{X^i\cup X^{i+1} \leftarrow i}.$  Then $\Gamma^0 + \Gamma^1 + \Gamma^2$ is a legal sequence of moves starting from $\eta^{X^0\leftarrow i}$ that ends at $\eta^{X^0\cup X^1\cup X^2 \cup X^3 \leftarrow i}.$  The reverse of $\Gamma^0+\Gamma^1$ is legal starting from this configuration, ending at $\eta^{X^0 \cup X^3 \leftarrow i}$.  

Repeat a similar argument through the bricks $\clean$.  For $i=0,\cdots ,6$, let $Y^i = R^i \cap \bar{R}^{i+1}$ and let $Y^7=R^7 \cap S^0$.  There is a legal sequence starting from $\eta^{X^0\cup X^3 \leftarrow i}$ that ends at $\eta^{X^0\cup X^3\cup Y^7 \leftarrow i}$.  The sail $R^7$ is a separating set that lies in section 0 of $B^0$.  By Corollary \ref{mrclean}, there exists a sequence of legal moves restricted to the sites in $B^0$ that begins at $\eta^{X^0 \cup X^3 \cup Y^7 \leftarrow i}$ and ends at $\eta^{X^3\cup Y^7 \leftarrow i}.$  The reverse of the sequence that infected $Y^7$ from $X^3$ can now be applied to $\eta^{X^3 \cup Y^7 \leftarrow i}$ to end at the desire configuration $\eta^{X^3 \leftarrow i}$.  

The same argument applies to $\trans'$, using the same set $\clean.$

Altogether the translation of infection to $\eta^{X^3\leftarrow i}$ took at most $O(L^2)$ where each step in the configuration differed from $\eta$ on a set of size $O(L)$.     
\end{proof}

Analogous to the two-dimensional case, we define good paths of bricks and super-good paths of bricks along which infections will propagate.

\begin{defn}
Fix a configuration $\eta$.  Consider a possible bi-infinite sequence in $\zz^3$, denoted by $\cdots z_{-1}, z_0, z_1,\cdots$.  We say that this sequence defines a \emph{good path of bricks} if the following hold:

\begin{enumerate}

\item $z_i + B$ is good,
\item for all $C\in \clean,$ $z_i + C$ is good
\item one of the two following conditions hold:
	\begin{itemize}
		\item either $z_{i+1} + B = z_i + B^3$ and $z_i + B^j$ are good for $j=0,1,2,3.$
		\item or $z_{i+1} +B = z_i + A^3$ and $z_i + A^j$ are good for $j = 0,1,2,3.$
	\end{itemize}
	
\end{enumerate}

\end{defn}

\begin{defn}
Fix a configuration $\eta$ and a good path of bricks $z_{-l}, \cdots z_0$.  We say this path is \emph{super-good} if for some $i$, the brick $z_i + B$ has a sail whose bottom layer is completely infected.  
\end{defn}

\begin{prop}\label{prop:origin_infected}
	Fix a configuration $\eta$.  Let $z_{-l} \cdots z_0 = -(L+1,4L+4,16L)$, be a super good path.  Then there exists a sequence legal moves $\Gamma$ of length $N = O(L^2l)$ such that
	
	\begin{enumerate}
\item $\Gamma^{N}(\eta)(0) = 0$, i.e. the origin is infected the final configuration,
\item for all $i\le N$, $\Gamma^i(\eta)$ differs from $\eta$ on a set of size is at most $O(L)$, 
\item for all $i\leq N$, $\Gamma^i(\eta)$ agrees with $\eta$ on all sites outside $x_i + [O(L)]^2$, where $x_i$ is the location of the $i$th move.
\end{enumerate}
	    
\end{prop}

\begin{proof}

This follows from Lemma \ref{allinfamily} by propagating the empty line guaranteed by the path being super-good. 
\end{proof}

Similar to the two-dimensional case, we need to define the scaling.  Let $L = 2\log(1/q)/q$, $l = q^{-L-1}$.  We also define the following events:
\begin{itemize}
\item $E := \{$ there exists super-good path $z_{-l}\cdots z_0 = -(L+1,4L+4,16L)\}$
\item $A := \{ \eta(0) = 0\}$.
\end{itemize}

All that remains to properly define the event $E$ in the case of polluted $\zz^3$, and show that it has high probability.

\begin{lem}\label{goodpathprob}
	The probability that there exists a bi-infinite good path containing $z_0 = -(L+1,4L+4, 16L)$ tends to 1 as $(\pi,q)\to (0,0).$  
\end{lem}

\begin{proof}

The probability for a fixed brick to be good tends to 1 as $(\pi,q)\to (0,0).$  Properties 1-5 in the definition of a good brick a satisfied with high probability according to \cite[Proposition 6]{GravnerHolroyd_polluted} while the last property holds with high probability through an argument similar to that found in the proof of Claim \ref{claim:fa2polluted_goodprob}.

Construct the lattice of bricks that are translates of $z_0 + B$.  Two bricks, $z+B$ and $z'+B$, will be connected by an edge if $z,z'$ is a good path of bricks of size two.  This induces an directed edge-percolation process on a two dimensional lattice.  This process has finite range dependencies, thus by LSS \cite{LSS} and the fact that oriented percolation contains an infinite cluster when the percolation parameter is large enough, there exists a bi-infinite good path containing $z_0$ with probability that tends to $1$ as $(\pi,q)\to (0,0)$.   
\end{proof}

\begin{lem} \label{soop_is_on}
The probability that there exists a super-good path of length $l$ ending at $z_0 = -(L+1,4L+4,16L)$ tends to $1$ as $(\pi,q)\to(0,0)$.  	
\end{lem}

\begin{proof}

This follows the same argument as in the proof of Claim \ref{claim:fa2polluted_supergoodifgood}.	
\end{proof}

The proof of Theorem \ref{thm:z3} is direct application of Proposition \ref{prop:yoyoma}, Lemma \ref{soop_is_on}, and Proposition \ref{allinfamily}.   

\section{Questions}
\begin{itemize}
\item Match the upper and lower bound in Theorem \ref{thm:z3}.
\item Can the methods of \cite{GravnerHolroydSivakoff_polluted3d} be used in order to analyze the FA$3$f model on polluted $\zz^3$? More generally, what can we say about the Fredrickson-Andersen model with general threshold in general dimension?
\item The methods introduced above are rather soft, in the sense that they only require some high probability event from which a path that empties the origin could be constructed. It is therefore plausible that they could be applied for models in which other, stronger techniques, fail, e.g., KCMs on finite graphs (whose dynamics is not ergodic).
\item Can we analyze other time scales of the system, and in particular the typical time scale in which time correlation of local funtcions is lost (see, e.g., \cite[section 3.6]{Shapira_thesis})?

\end{itemize}

\section*{Acknowledgements}
We would like to Cristina Toninelli for intoducting us to this problem and to each other, as well as for the useful discussion.

\bibliographystyle{plain}
\bibliography{fa2f_polluted}

\end{document}